\documentclass{amsart}
\usepackage{amsopn, amsthm}
\usepackage{amssymb, amscd}
\usepackage{graphicx}
\newcommand{\nc}{\newcommand}

\nc{\fg}{\mathfrak{f} } \nc{\vg}{\mathfrak{v} } \nc{\wg}{\mathfrak{w} }
\nc{\zg}{\mathfrak{z} } \nc{\ngo}{\mathfrak{n} } \nc{\kg}{\mathfrak{k} }
\nc{\mg}{\mathfrak{m} } \nc{\bg}{\mathfrak{b} } \nc{\ggo}{\mathfrak{g} }
\nc{\ggob}{\overline{\mathfrak{g}} } \nc{\sog}{\mathfrak{so} }
\nc{\sug}{\mathfrak{su} } \nc{\spg}{\mathfrak{sp} } \nc{\slg}{\mathfrak{sl} }
\nc{\glg}{\mathfrak{gl} } \nc{\cg}{\mathfrak{c} } \nc{\rg}{\mathfrak{r} }
\nc{\hg}{\mathfrak{h} } \nc{\tgo}{\mathfrak{t} } \nc{\ug}{\mathfrak{u} }
\nc{\dg}{\mathfrak{d} } \nc{\ag}{\mathfrak{a} } \nc{\pg}{\mathfrak{p} }
\nc{\sg}{\mathfrak{s} } \nc{\affg}{\mathfrak{aff} } \nc{\qg}{\mathfrak{q} }

\nc{\pca}{\mathcal{P}} \nc{\nca}{\mathcal{N}} \nc{\lca}{\mathcal{L}}
\nc{\oca}{\mathcal{O}} \nc{\mca}{\mathcal{M}} \nc{\tca}{\mathcal{T}}
\nc{\aca}{\mathcal{A}} \nc{\cca}{\mathcal{C}} \nc{\gca}{\mathcal{G}}
\nc{\sca}{\mathcal{S}} \nc{\hca}{\mathcal{H}} \nc{\bca}{\mathcal{B}}
\nc{\fca}{\mathcal{F}} \nc{\dca}{\mathcal{D}} \nc{\val}{\operatorname{val}}
\nc{\eca}{\mathcal{E}}

\nc{\vp}{\varphi} \nc{\ddt}{\tfrac{{\rm d}}{{\rm d}t}}
\nc{\dpar}{\tfrac{\partial}{\partial t}} \nc{\im}{\mathtt{i}}

\nc{\SO}{\mathrm{SO}} \nc{\Spe}{\mathrm{Sp}} \nc{\Sl}{\mathrm{SL}}
\nc{\SU}{\mathrm{SU}} \nc{\Or}{\mathrm{O}} \nc{\U}{\mathrm{U}} \nc{\Gl}{\mathrm{GL}}
\nc{\Se}{\mathrm{S}} \nc{\Cl}{\mathrm{Cl}} \nc{\Spein}{\mathrm{Spin}}
\nc{\Pin}{\mathrm{Pin}} \nc{\G}{\mathrm{GL}_n(\RR)} \nc{\g}{\mathfrak{gl}_n(\RR)}
\nc{\RR}{{\Bbb R}} \nc{\HH}{{\Bbb H}} \nc{\CC}{{\Bbb C}} \nc{\ZZ}{{\Bbb Z}}
\nc{\FF}{{\Bbb F}} \nc{\NN}{{\Bbb N}} \nc{\QQ}{{\Bbb Q}} \nc{\PP}{{\Bbb P}}
\nc{\SSS}{{\Bbb S}} \nc{\DD}{{\Bbb D}}

\nc{\vs}{\vspace{.2cm}} \nc{\vsp}{\vspace{1cm}} \nc{\ip}{\langle\cdot,\cdot\rangle}
\nc{\ipp}{(\cdot,\cdot)} \nc{\la}{\langle} \nc{\ra}{\rangle} \nc{\unm}{\tfrac{1}{2}}
\nc{\unc}{\tfrac{1}{4}} \nc{\und}{\tfrac{1}{16}} \nc{\no}{\vs\noindent}
\nc{\lamn}{\Lambda^2(\RR^n)^*\otimes\RR^n} \nc{\lamp}{\Lambda^2\pg^*\otimes\pg}
\nc{\lamg}{\Lambda^2\ggo^*\otimes\ggo} \nc{\lamngo}{\Lambda^2\ngo^*\otimes\ngo}
\nc{\tangz}{{\rm T}^{\rm Zar}} \nc{\mum}{/\!\!/} \nc{\kir}{/\!\!/\!\!/}
\nc{\Ri}{\tfrac{4\Ric_{\mu}}{||\mu||^2}} \nc{\ds}{\displaystyle}
\nc{\ben}{\begin{enumerate}} \nc{\een}{\end{enumerate}} \nc{\f}{\frac}
\nc{\lb}{[\cdot,\cdot]} \nc{\isn}{\tfrac{1}{||v||^2}}
\nc{\gkp}{(\ggo=\kg\oplus\pg,\ip)} \nc{\ukh}{(\ug=\kg\oplus\hg,\ip)}

\nc{\Hess}{\operatorname{Hess}} \nc{\ad}{\operatorname{ad}}
\nc{\Ad}{\operatorname{Ad}} \nc{\rank}{\operatorname{rank}}
\nc{\Irr}{\operatorname{Irr}} \nc{\End}{\operatorname{End}}
\nc{\Aut}{\operatorname{Aut}} \nc{\Inn}{\operatorname{Inn}}
\nc{\Der}{\operatorname{Der}} \nc{\Ker}{\operatorname{Ker}}
\nc{\Iso}{\operatorname{I}} \nc{\Diff}{\operatorname{Diff}}
\nc{\Lie}{\operatorname{L}} \nc{\tr}{\operatorname{tr}} \nc{\dif}{\operatorname{d}} \nc{\Hom}{\operatorname{Hom}}
\nc{\sen}{\operatorname{sen}} \nc{\modu}{\operatorname{mod}}
\nc{\Riem}{\operatorname{Rm}} \nc{\Ricci}{\operatorname{Ric}}
\nc{\sym}{\operatorname{sym}} \nc{\symac}{\operatorname{sym^{ac}}}
\nc{\symc}{\operatorname{sym^{c}}} \nc{\scalar}{\operatorname{sc}}
\nc{\grad}{\operatorname{grad}} \nc{\ricci}{\operatorname{ric}}
\nc{\nr}{\operatorname{nr}} \nc{\riccic}{\operatorname{ric^{c}}}
\nc{\riccig}{\operatorname{ric^{\gamma}}} \nc{\Rin}{\operatorname{M}}
\nc{\Le}{\operatorname{L}} \nc{\tang}{\operatorname{T}}
\nc{\level}{\operatorname{level}} \nc{\rad}{\operatorname{r}}
\nc{\abel}{\operatorname{ab}} \nc{\CH}{\operatorname{CH}}
\nc{\mcc}{\operatorname{mcc}} \nc{\Adj}{\operatorname{Adj}}
\nc{\Order}{\operatorname{O}}

\def\be{\begin{equation*}} \def\ee{\end{equation*}}
\def\bea{\begin{eqnarray*}} \def\eea{\end{eqnarray*}}
\def\bpr{\begin{proposition}} \def\epr{\end{proposition}}
\def\bsm{\begin{smallmatrix}} \def\esm{\end{smallmatrix}}
\nc{\Liealg}{\operatorname{Lie}} \nc{\diag}{\operatorname{diag}}
\nc{\Gr}{\operatorname{Gr}}
\nc{\so}{\mathfrak{so} }

\theoremstyle{plain}
\newtheorem{theorem}{Theorem}[section]
\newtheorem{proposition}[theorem]{Proposition}

\newtheorem{lemma}[theorem]{Lemma}
\newtheorem{problem}[theorem]{Problem}

\theoremstyle{definition}
\newtheorem{definition}[theorem]{Definition}

\theoremstyle{remark}

\newtheorem{remark}[theorem]{Remark}
\newtheorem{example}[theorem]{Example}

\title{Solsolitons associated with graphs}

\author{Ramiro A. Lafuente}

\address{FaMAF and CIEM, Universidad Nacional de C\'ordoba, C\'ordoba, Argentina}
\email{ramlaf@gmail.com}

\begin{document}

\maketitle
\pagestyle{myheadings}
\markboth{\textsc{Solsolitons associated with graphs}}{\textsc{Ramiro A. Lafuente}}

\begin{abstract}
We show how to associate with each graph with a certain property (\emph{positivity}) a family of simply connected solvable Lie groups endowed with left-invariant Riemannian metrics that are Ricci solitons (called \emph{solsolitons}). We classify them up to isometry, obtaining families depending on many parameters of explicit examples of Ricci solitons. A classification of graphs with up to 3 coherent components according to positivity is also given.
\end{abstract}

%\newpage
%\tableofcontents

%\mainmatter

%\newpage
%---------------------------------------------------------------------------------------------------------------------
\section{Introduction}\label{intro}
An important conceptual contribution of Ricci flow theory has been the notion of Ricci solitons, which generalize Einstein metrics, and are precisely the fixed points of the flow up to scaling and pull-back by diffeomorphisms (see \cite[Ch. I]{libro} for further information).

In the homogeneous case, all known examples of nontrivial Ricci solitons are {\it solsolitons}, i.e. simply connected solvable Lie groups endowed with a left invariant metric $g$ whose Ricci operator satisfies $\Ricci(g)=cI+D$, for some $c\in\RR$ and derivation $D$ of the Lie algebra.  Such metrics are called {\it nilsolitons} in the nilpotent case, and have been extensively studied because of their interplay with Einstein solvmanifolds (see the survey \cite{cruzchica}). It was recently proved in \cite{solsolitons} that, up to isometry, any solsoliton can be obtained via a very simple
construction from a nilsoliton together with any abelian Lie algebra of
symmetric derivations of its metric Lie algebra. The following
uniqueness result is also obtained in \cite{solsolitons}: a given solvable Lie group can admit at most one
solsoliton up to isometry and scaling. In this way, solsolitons provide canonical metrics on solvable Lie groups, where Einstein metrics may not exist.  In \cite{wi}, a classification of solsolitons in dimension $\leq 7$ is given.

In this paper, we consider certain nilsolitons attached to graphs found in \cite{LauWill} and apply the construction mentioned above to obtain families, depending on many parameters, of explicit examples of solsolitons.  Given a graph $\gca$, a $2$-step nilpotent Lie algebra $\ngo_\gca$ can be naturally defined, which admits a nilsoliton metric if and only if $\gca$ is {\it positive}; namely,
a certain uniquely defined weighting on the graph is positive.  For example, graphs with at most five vertices (excepting only one), regular graphs and trees without any edge adjacent to four or more edges are positive (see \cite{LauWill}). We prove that for any positive graph $\gca$ with $p$ vertices and $q$ edges, the set of $(r+p+q)$-dimensional solsolitons associated to $\ngo_\gca$, up to isometry and scaling, is parameterized by the quotient
$$
\Gr_r(\RR^p)/\Aut(\gca),
$$
where $\Gr_r(\RR^p)$ is the Grassmannian of $r$-dimensional subspaces of $\RR^p$ and $\Aut(\gca)$ is the automorphism group of the graph $\gca$ acting naturally as permutations on $\RR^p$ (see Theorem \ref{TeoMain} for a more precise statement).  In this way, as $\Aut(\gca)$ is a finite group, one obtains a family depending on $r(p-r)$ parameters of solsolitons on each dimension $r+p+q$. There is a single element in this family which is Einstein.

In Section \ref{ejemplos}, we exhibit new examples of positive and nonpositive graphs by giving a complete classification for graphs having at most three coherent components.

%\newpage
%---------------------------------------------------------------------------------------------------------------------
\section{Preliminaries}\label{preli}

\vskip10pt

\subsection{Homogeneous Ricci solitons}

We begin by giving a brief review about Ricci solitons on homogeneous manifolds, following \cite{solsolitons}.

A complete Riemannian metric $g$ on a differentiable manifold $M$ is said to be a \emph{Ricci soliton} if its Ricci tensor $\Ricci(g)$ satisfies
\begin{equation}\label{Ricci_soliton}
\Ricci(g) = cg + \Le_X \, g, \,\,\,\,\,\,\, \hbox{for some } c\in \RR, \,\, X \in \chi(M) \hbox{ complete,}
\end{equation}
where $\Le_X$ is the usual Lie derivative, and $\chi(M)$ is the space of all differentiable vector fields on $M$.
%If furthermore $X = \grad f$ for some $f\in \cca^{\infty}(M)$, then (\ref{Ricci_soliton}) becomes $\Ricci(g) = cg + 2\Hess(f)$ and $g$ is called a \emph{gradient} Ricci soliton.

We see that Ricci solitons are very natural generalizations of Einstein metrics (i.e. $\Ricci(g) = cg$). However, the main significance of the concept is that $g$ is a Ricci soliton if and only if the curve of metrics in $M$
\begin{equation}\label{curva_metricas}
g(t) = (-2ct + 1) \phi^*_t g,
\end{equation}
is a solution to the Ricci flow
\begin{equation}
\frac{\partial}{\partial t} g(t) = -2 \Ricci(g(t)),
\end{equation}
for some one-parameter group $\phi_t$ of diffeomorphisms of $M$.
%In other words, the Ricci flow starting at $g$ stays `forever' in the space of metrics which are homotetic to $g$; it is unable to `improve' $g$. Acording  to (\ref{curva_metricas}), Ricci solitons are called \emph{expanding}, \emph{steady} or \emph{shrinking}, depending on whether $c<0, c=0$ o $c>0$.

In this work we will only be interested in \emph{homogeneous} Ricci solitons (i.e. those defined on homogeneous manifolds).
\begin{remark}
It is worth pointing out that even Einstein metrics on homogeneous manifolds are still not completely understood. For a recent account of the theory we refer the reader to \cite{cruzchica}.
\end{remark}

As we have mentioned, all known examples so far of nontrivial homogeneous Ricci solitons are isometric to a left-invariant metric $g$ on a simply connected Lie group $G$, and can be obtained in the following way: we identify $g$ with an inner product on the Lie algebra $\ggo$ of $G$, and let us suppose that
\begin{equation}\label{cond_soliton}
\Ricci(g) = c I + D, \,\,\,\,\,\,\,\hbox{for some} \,\, c\in \RR, \,\, D \in \Der(\ggo),
\end{equation}
where $Ric(g)$ here also denotes the Ricci operator of $g$ (i.e. $\Ricci(g) = g(\Ricci(g) \cdot, \cdot)$). Then $g$ is a Ricci soliton (see \cite{solsolitons} for more details).

\begin{remark}
It is an open question whether every left-invariant Ricci soliton metric satisfies (\ref{cond_soliton}). And concerning existence, up to now all nontrivial examples are on solvable Lie groups.
\end{remark}

\begin{definition}\cite{solsolitons}
A left-invariant metric $g$ on a simply connected solvable (nilpotent) Lie group is called a \emph{solsoliton} (\emph{nilsoliton}) if the corresponding Ricci operator satisfies (\ref{cond_soliton}).
\end{definition}
We will usually identify such a metric with the corresponding inner product on the Lie algebra of the group.

Recall that a simply connected Lie group endowed with a left-invariant Riemannian metric is called a \emph{solvmanifold} if it is solvable, and \emph{nilmanifold} if it is nilpotent.

\vskip15pt

\subsection{Solsolitons built up from nilsolitons}\label{sols from nils}

The following construction from \cite{solsolitons} shows how, in a very natural way, solsolitons can be obtained from a nilsoliton.
\begin{proposition}\cite{solsolitons}\label{constr_solsolitons}
Let $(\ngo,\ip_1)$ be a nilsoliton Lie algebra, say with Ricci operator $\Ricci_1=cI+D_1$,
$c<0$, $D_1\in\Der(\ngo)$, and consider $\ag$ any abelian Lie algebra of $\ip_1$-symmetric
derivations of $(\ngo,\ip_1)$.  Then the solvmanifold $S$ with Lie algebra
$\sg=\ag\oplus\ngo$ (semidirect product) and inner product given by
$$
\ip|_{\ngo\times\ngo}=\ip_1, \qquad \la\ag,\ngo\ra=0, \qquad \la
A,A\ra=-\tfrac{1}{c}\tr{A^2}, \,\,\forall A\in\ag,
$$
is a solsoliton, with $\Ricci=cI+D$, for some $D\in\Der(\sg)$.
\end{proposition}

From now on, when there is no ambiguity, we will speak about nilsolitons or solsolitons instead of nilmanifolds or solvmanifolds, often referring to the Lie algebra of the group endowed with the corresponding inner product. We see that the construction above may lead to different solsolitons, depending on the algebra of derivations $\ag$ chosen. Two questions arise naturally at this point: whether every solsoliton can be constructed in this way, and when two of them are isometric.

The answer to the first question is known to be affirmative. The following result is given as a corollary to the structure theorem on solsolitons \cite[Theorem 4.8]{solsolitons}:
\begin{theorem}\cite{solsolitons}
Up to isometry, any solsoliton can be constructed as in Proposition \ref{constr_solsolitons}.
\end{theorem}

Concerning the second question, we cite the following proposition, which gives us the equivalence relation to be considered in order to study solsolitons up to isometry.
\begin{proposition}\cite{solsolitons}\label{EquivSolsolitons}
Let $(\ngo, \la \cdot, \cdot \ra_1 ) $ be a nilsoliton and let us consider two solsolitons $S$ and $S'$ constructed as in Proposition \ref{constr_solsolitons} for abelian Lie algebras
\be
\ag, \ag' \subseteq \Der(\ngo)\cap \sym(\ngo, \ip_1 ),
\ee
respectively. Then $S$ is isometric to $S'$ if and only if there exists $h \in \Aut(\ngo)\cap \Or(\ngo, \ip_1 )$ such that $\ag' = h \ag h^{-1}$.
\end{proposition}

\vskip15pt

\subsection{The Lie algebra associated with a graph}
Here we show how to associate a Lie algebra to a graph, and we describe conditions on the graph such that the resulting Lie algebra admits a nilsoliton inner product.

Let $\gca = (S,E)$ be a (finite, undirected) graph, with set of vertices $S = \{v_1, \ldots ,v_p  \}$ and edges $E = \{l_1, \ldots, l_q \}$, $l_k = v_i v_j$ for some $i,j$. We associate to it the Lie algebra $\ngo_{\gca} = (\RR^n, [ \cdot , \cdot ] )$, $n = p+q$, with
\be
[e_i,e_j] = \left\{
              \begin{array}{ll}
                e_{p+k}, & \hbox{if $l_k = v_i v_j$, $i<j$;} \\
                0, & \hbox{otherwise $(i<j)$,}
              \end{array}
            \right.
\ee
where $\{e_i\}_{i=1}^n$ is the standard basis for $\RR^n$. We will often identify the vertices of the graph with the vectors $e_i$'s, $1\leq i \leq p$, and the edges with the $e_k$'s, $p+1\leq k \leq p+q$. Then the bracket between two vertices $v_i$ and $v_j$ with $i<j$ is the edge that joins them, if it exists, and zero otherwise. To obtain a well defined bracket we add the assumption that no two edges join the same pair of vertex.

This construction was first considered in \cite{DM}, where the automorphisms group of the Lie algebra is studied, and then in \cite{LauWill}, where necessary and sufficient conditions for the Lie algebra to admit a nilsoliton inner-product are obtained.

The Lie algebra $\ngo_{\gca}$ so obtained is clearly 2-step nilpotent, and it was proved by M. Mainkar that it is closely related to the starting graph, in the sense that two $\ngo_{\gca}$ of those are isomorphic as Lie algebras if and only if the corresponding graphs are isomorphic as graphs.

\begin{remark}
The mapping that takes graphs onto 2-step nilpotent Lie algebras is not surjective: there are for instance continuous families in the space of 2-step nilpotent Lie algebras up to isomorphism (starting with dimension $9$), whereas there are only finitely many of the graph-kind in each dimension.
\end{remark}

The following result describes the condition that a graph Lie algebra must satisfy to admit a nilsoliton inner product, in terms of the graph.
\begin{proposition}\cite{LauWill}
 $\ngo_{\gca}$ admits a nilsoliton inner product if and only if there exist weights $c_1, \ldots, c_q \in \RR$ for the edges such that
\begin{eqnarray*}
 3 c_k + \sum_{l_m \sim l_k} c_m = \nu, \,\,\,\,\, & &\forall k = 1,\ldots, q, \\
 c_k > 0, \,\,\,\,\,& & \forall k = 1,\ldots, q,
\end{eqnarray*}
where the first sum is over all edges $l_m$ that share a vertex with $l_k$, and $\nu >0$ is any real number.
\end{proposition}
A graph satisfying this property is called \emph{positive}. If we consider the line graph $L(\gca)$ of $\gca$, the first condition above may be written in terms of its adjacency matrix $\Adj L(\gca)$ as
\begin{equation}\label{positivo}
(3I + \Adj L(\gca)) \left[ \begin{smallmatrix} c_1\\ \vdots \\ c_q  \end{smallmatrix} \right] =
\nu \left[ \begin{smallmatrix} 1\\ \vdots \\ 1  \end{smallmatrix} \right]
\end{equation}
It can be proved that the matrix $3I + \Adj L(\gca)$ is positive definite, thus given $\nu >0$ the above system has always a unique solution. And since $\nu>0$, the condition about the positivity of the $c_j$ is just to say that the following vector
\be
(3I + \Adj L(\gca))^{-1} \left[ \begin{smallmatrix} 1\\ \vdots \\ 1  \end{smallmatrix} \right]
\ee
has positive entries.

Let us see now, for further reference, how to calculate the nilsoliton inner product on $\ngo_{\gca}$. Let $(c_k)$ be the solution to (\ref{positivo}) with  $c_k>0$ $\forall k$ and $\sum_{k=1}^q c_k = 1$. Let us define the operator $g$ in $\ngo_{\gca}$ whose matrix in the standard basis is
\begin{equation}\label{def_g}
g =
\left[
  \begin{array}{cc}
    I_p & 0 \\
    0 & C \\
  \end{array}
\right], \,\,\,\, \hbox{  }
C =
\left[
  \bsm
    \sqrt{c_1}  &   & 0\\
       & \ddots & \\
    0  &   & \sqrt{c_q} \\
  \esm
\right] \in \RR^{q\times q}, \,\,\,
I_p =
\left[
  \bsm
    1  &   & 0\\
       & \ddots & \\
    0  &   & 1 \\
  \esm
\right]
\in \RR^{p\times p}.
\end{equation}
Then the nilsoliton inner product $\ip_1$ on $\ngo_{\gca}$ is given by the formula
\begin{equation}\label{pinilsoliton}
\la z, w \ra_1 = \la g^{-1} z, g^{-1}w \ra, \quad \forall z,w \in \ngo_{\gca}
\end{equation}
where $\ip$ is the canonical inner product on $\RR^n$; see \cite{LauWill} for further details.

The standard basis $e_1, \ldots, e_n$ is orthogonal (although not necesarily orthonormal) with respect to this new inner product. If $V$ is the vector space generated by the vertices of $\gca$ (i.e. by $e_1, \ldots, e_p$), we can easily see that the nilsoliton restricted to $V\times V$ is precisely the canonical inner product there. From now on, for every subspace of $\ngo_{\gca}$ we will call $\ip$ the canonical inner product there, to distinguish from the restriction of $\ip_1$ to that subspace.

\begin{remark}\label{auto_grafo_preserva}
From the uniqueness of the nilsoliton weights, it is easy to see that if $\sigma : \gca \rightarrow \gca$ is a graph automorphism of $\gca$, then the weights that make the graph positive are invariant by $\sigma$ (i.e. if $l_k = v_i v_j$ and $l_m = \sigma(v_i) \sigma(v_j)$, then $c_k = c_m$).
\end{remark}

\vskip15pt

\subsection{Coherent decomposition of a graph}\label{CompCoherentes}
Here we introduce the notion of coherent components of a graph, which will be very useful when studying the Lie algebra associated with it. Let $\gca = (S,E)$ be a graph, and let us define for each $\alpha \in S$
\be \Omega'(\alpha) = \{ \omega \in S : \omega \alpha \in E \} \,\,\, \hbox{ and } \,\,\, \Omega(\alpha) = \Omega'(\alpha) \cup \{\alpha\}.  \ee
Now consider the equivalence relation $\sim$ in  $S$ defined as follows:
\be \alpha \sim \beta  \,\,\,\, \hbox{ if and only if }  \,\,\,\, \Omega'(\alpha) \subseteq \Omega(\beta) \hbox{ and } \Omega'(\beta) \subseteq \Omega(\alpha),  \ee
i.e. two vertices are related by $\sim$ if and only if they have 'the same neighbors'. Let $\Lambda = \Lambda(S,E)$ be the set of equivalence classes in $S$ with respect to $\sim$; for each $\lambda\in \Lambda$ we call $S_{\lambda} \subseteq S$ its equivalence class. The subsets $S_{\lambda}, \lambda \in \Lambda$, are the \emph{coherent components} of $(S,E)$; they form a partition of the set $S$.

This decomposition was considered in \cite{DM}, where the following properties are also mentioned:

$\bullet$ If the graph $(S,E)$ is not connected and has no isolated vertices, its coherent components are just the set of the coherent components from each of its connected components. On the other hand, if it does have isolated vertices, all of them conform one coherent component, and the rest of the graph is decomposed as above.

$\bullet$ Given $\gca = (S,E)$, with $S_{\lambda}, \lambda \in \Lambda$ its coherent components, it is easy to see that if for a given $\lambda\in \Lambda$ there exist $\alpha, \beta \in S_{\lambda}$ such that $\alpha \beta \in E$, then $\xi \eta \in E$ $\forall \xi, \eta \in S_{\lambda}$. This implies that a coherent component is on its own either a complete graph or a discrete one.

$\bullet$ To generalize the previous item let us assume that, given $\lambda, \mu \in \Lambda$, there exist $\alpha \in S_{\lambda}$ and $\beta \in S_{\mu}$ such that $\alpha \beta \in E$. Then it is easy to see that $\xi \eta \in E$ for all $\xi \in S_{\lambda}$, $\eta \in S_{\mu}$. Therefore, given two coherent components $S_{\lambda}$ and $S_{\mu}$, either they are not adjacent at all, or every possible edge between them is present in $E$. To sum up, let us define a set of unordered pairs $\eca$ in such a way that $\lambda \mu \in \eca$ if and only if the components $S_{\lambda}$ and $S_{\mu}$ are 'adjacent'. We call $(\Lambda, \eca)$ the \emph{coherence graph} associated with $(S,E)$. Let $\Lambda_0 \subseteq \Lambda$ be the subset of coherent components $S_{\lambda}$ such that $(S,E)$ restricted to it is a complete graph, and furthermore let us consider a set $M = \{ m_{\lambda} : \lambda \in \Lambda\} \subseteq \NN$. From the properties mentioned above it is clear that the coherence graph $(\Lambda, \eca)$ together with $\Lambda_0$ and $M$ fully determine (up to isomorphism) the graph $(S,E)$: given $\alpha, \beta \in S$, say $\alpha \in S_{\lambda}$, $\beta \in S_{\mu}$, $\alpha \beta \in E$ if and only if $\lambda = \mu \in \Lambda_0$ or $\lambda \mu \in \eca$, and given $\lambda \in \Lambda$, the correspondent coherent component $S_{\lambda}$ is a graph with $m_{\lambda}$ vertices that is complete if $\lambda \in \Lambda_0$, and discrete otherwise.

These properties give us the following useful result on the weights of a positive graph. We call two edges \emph{similar} if they join the same pair of coherent components.
\begin{proposition}\label{mismos_pesos}
Let $\gca = (S,E)$ be a positive graph, with weights $(c_i)_{i=1}^q$ for some $\nu>0$ fixed. If $l_i, l_j$ are two similar edges, then $c_i = c_j$.
\end{proposition}
\begin{proof}
By Remark \ref{auto_grafo_preserva} it suffices to prove that there is a graph automorphism interchanging $l_i$ and $l_j$. Let us write $l_i = v_a v_b$, $l_j = v_c v_d$, with $v_a,v_c \in S_{\lambda}$, $v_b, v_d \in S_{\mu}$. We define $\sigma: S \rightarrow S$ by $\sigma(v_a) = v_c$, $\sigma(v_c) = v_a$, $\sigma(v_b) = v_d$, $\sigma(v_d) = v_b$, and $\sigma(v) = v$ otherwise. From the properties of the coherent components it is clear that $\sigma$ is a graph automorphism of $\gca$, and it interchanges $l_i$ with $l_j$.
\end{proof}

%\begin{remark}\label{mismos_pesos}
%It is clear from the definitions that given two edges that join the same pair of coherent components, there exists a graph automorphism that swaps them. Thanks to Remark \ref{auto_grafo_preserva}, if $\gca$ is positive then we conclude that such edges must have the same weights. For example, the edges of a complete coherent component all weight the same. This has further implications, mainly in Section \ref{ejemplos}. Also, together with formula (\ref{pinilsoliton}) this gives us a lot of information about the nilsoliton inner product $\ip_1$.
%\end{remark}

%\newpage
%---------------------------------------------------------------------------------------------------------------------
\section{Symmetric derivations and orthogonal automorphisms of $\ngo_{\gca}$}\label{herra}

We have shown how to associate solsolitons to a positive graph. Following the discussion in \ref{sols from nils}, it is clear that in order to classify solsolitons up to isometry we have to solve the following problem:
\begin{problem}\label{problema}
To describe the set of all abelian Lie algebra of symmetric derivations of $\ngo_{\gca}$ up to conjugation by an orthogonal automorphism of $\ngo_{\gca}$.
\end{problem}
Recall that the notions of symmetric and orthogonal here are with respect to the nilsoliton inner product $\ip_1$ on $\ngo_{\gca}$.

\vskip15pt

\subsection{Derivations of a 2-step nilpotent Lie algebra}
Let $\ngo$ be an arbitrary 2-step nilpotent Lie algebra, and let us consider a vector space $V$ complementary to $[\ngo,\ngo]$.
%In \cite{DM} the Lie group $\Aut(\ngo)$ is studied, and a decomposition according to the decomposition $\ngo = V \oplus [\ngo,\ngo]$ is obtained for it. Since $\Der(\ngo)$ is the Lie algebra of $\Aut(\ngo)$, it is natural to expect a similar decomposition for this algebra.

Given $\theta \in \Hom(V,[\ngo,\ngo])$ we define $D_{\theta} \in \End(\ngo)$ as $D_{\theta}(\xi) = \theta (\pi (\xi))$, where $\pi : \ngo \rightarrow V$ is the canonical projection. Since $\ngo$ is 2-step nilpotent, we see that $D_{\theta}([\xi,\eta]) = \theta(\pi([\xi,\eta])) = \theta(0) = 0 = [D_{\theta}\xi,\eta] + [\xi,D_{\theta}\eta]$, then $D_{\theta}\in \Der(\ngo)$. These derivations form a Lie subalgebra of $\Der(\ngo)$, which we call $\ug$. Moreover, let $\tgo = \{D\in \Der(\ngo)\, : \, D(V) \subseteq V \}$, another Lie subalgebra of $\Der(\ngo)$.
\begin{proposition}\label{descDer}
$\Der(\ngo) = \ug \oplus \tgo$, as vector spaces.
\end{proposition}
\begin{proof}
It is clear that $\ug \cap \tgo = 0$. Now if $D\in \Der(\ngo)$, let $\phi \in \Hom(V,V)$, $\psi\in \Hom(V,[\ngo,\ngo])$ such that $D(v) = \phi(v) + \psi(v)$ for all $v\in V$. Let $D_{\psi}\in \ug$ be the derivation associated with $\psi$. Then $D-D_{\psi} \in \Der(\ngo)$ and $(D-D_{\psi})(v) = \phi(v) \in V$, hence $D-D_{\psi}\in \tgo$.
\end{proof}

About $\tgo$ we can say that, since $[V,V] = [\ngo,\ngo]$, its elements are clearly determined by their restrictions to $V$, $\ggo = \{ D|_V : D\in \tgo \}$, a Lie subalgebra of $\glg(V)$. If we call $T$ the subgroup of automorphisms of $\Aut(\ngo)$ that leave $V$ invariant, these are also clearly determined by its restrictions to $V$. Let $G \subseteq GL(V)$ be the Lie subgroup consisting of restrictions of elements of $T$ to $V$; it is clear that $\ggo$ is the Lie algebra of $G$. We will mention some of the properties of this Lie algebra, which has been studied in \cite{DM}.

\subsection{Symmetric derivations of $\ngo_{\gca}$}
Suppose now that $\ngo_{\gca}$ is the Lie algebra associated with a positive graph $\gca = (S,E)$. In this case, $V$ becomes the vector space spanned by `the vertices', and $[\ngo_{\gca},\ngo_{\gca}] = [V,V]$ is the one spanned by `the edges'. We have the following information about $\ggo$.
\begin{proposition}\cite{DM}\label{ggo}
Let $\Lambda$ be the set of coherent components of the graph $(S,E)$, and for each $\lambda \in \Lambda$ let $S_{\lambda}$ be the corresponding coherent component. Finally, let $V_{\lambda}$ be the subspace of $V$ spanned by $S_{\lambda}$. Then, \be \ggo = \qg \oplus \mg \ee
where $\qg = \bigoplus_{\lambda \in \Lambda} \End(V_{\lambda})$, viewed as a Lie subalgebra of $\End(V)$ via the canonical embedding, and $\mg$ is a nilpotent ideal in $\ggo$. Furthermore, the elements of $\Lambda$ can be enumerated as $\lambda_1, \ldots, \lambda_k$ so that $\bigoplus_{i\leq j} V_{\lambda_i}$ is $D$-invariant, for each $D\in \ggo$ and each $j= 1, \ldots, k$.
\end{proposition}

Actually, we can see that $\mg \subseteq \bigoplus_{i<j} \Hom(V_{\lambda_j}, V_{\lambda_i})$. If we fix the canonical basis for $\ngo_{\gca}$, choosing an ordering for its elements so that vertices of the same coherent component are consecutive (and also taking into account the enumeration $\lambda_1, \ldots, \lambda_k$ from above), the proposition implies that the matrices of the elements of $\ggo$ are `block-triangular':
\begin{equation}\label{matrizggo}
\begin{array}{c}
    \lambda_1 \\
    \lambda_2 \\
    \lambda_3 \\
              \\
    \lambda_k \\
\end{array}
\left[
  \begin{array}{ccccc}
      A_1   &   *   &    *  & \dots &    *  \\
       0    &   A_2 &    *  & \dots &   *   \\
       0    &   0   &  A_3  & \dots &   *   \\
      \vdots& \vdots&\vdots & \ddots&\vdots \\
       0    &   0   &   0   & \dots &  A_k  \\
  \end{array}
\right]
\end{equation}
with $A_i \in \End(V_{\lambda_i})$, $i = 1, \ldots, k$. The blocks $A_i$ represent the component in $\qg$, and the $*$ the component in $\mg$.

Now take $D\in \Der(\ngo_{\gca})$ such that $D$ is $\ip_1$-symmetric. From Proposition \ref{descDer} we have that $D = D_u + D_t$, with $D_u = \left[ \bsm 0 & 0 \\ D_{21} & 0\\ \esm \right]\in \ug, \,\,D_t = \left[ \bsm D_{11} & 0 \\ 0 & D_{22}\\ \esm \right] \in \tgo$ as matrices, where the blocks are defined according to the decomposition $\ngo_{\gca} = V \oplus [V,V]$. Since the chosen basis is $\ip_1$-orthogonal, it is clear that $D_u = 0$. The component $D_t$ is determined by its restriction to $V$, $D_{11}$, and the fact that $\ip_1 |_V$ is the canonical inner product there, implies that $D_{11}$ is a symmetric matrix. Using that $D_{11}\in \ggo$, we look at its matrix as in (\ref{matrizggo})and  we see that it is block-diagonal, with $A_i \in \sym(V_{\lambda_i}, \ip)$ for all $i$. This suggests us the following theorem.

\begin{theorem}\label{TeoDerSim}
Let $V = \bigoplus_{\lambda \in \Lambda} V_{\lambda}$ be the decomposition of $V$ with respect to the coherent components of the graph. Every $\ip_1$-symmetric derivation $D$ of $\ngo_{\gca}$ leaves $V$ invariant, and is determined by its restriction $D|_V$. If we call $\ggo_{sym} = \{D|_V : D\in \Der(\ngo_{\gca})\cap \sym(\ngo_{\gca},\ip_1)\} \subseteq \End(V)$, then
\be \ggo_{sym} = \bigoplus_{\lambda \in \Lambda} \sym(V_{\lambda}),\ee
where the symmetry is defined according to the canonical inner product in $V_{\lambda}$, and $\sym(V_{\lambda})$ is viewed as a subspace of $\End(V)$ via the canonical embedding.
\end{theorem}
\begin{proof}
According to the previous discussion, we only need to prove that if $D_{11} \in \bigoplus_{\lambda \in \Lambda} \sym(V_{\lambda})$ and $D = \left[ \bsm D_{11} & 0 \\ 0 & D_{22}\\ \esm \right]$ is the derivation associated with it, then $D$ is $\ip_1$-symmetric, i.e.
\begin{equation}\label{Simetria}
 \la De_i, e_j \ra_1 = \la e_i, De_j \ra_1 \,\,\,\,\,\, \forall \,i,j \in \{1,\ldots, n\}.
\end{equation}
The properties we have mentioned about $\ip_1$ imply that (\ref{Simetria}) holds in the case $i,j \in \{1,\ldots,p\}$ (as $D_{11}$ is symmetric), and also if $i\in \{1,\ldots,p\}$, $j\in \{p+1,\ldots,n\}$ ($\ip_1$-orthogonality of the basis). It then suffices to prove (\ref{Simetria}) for $e_i, e_j$ such that $e_i = [e_a, e_b]$, $e_j = [e_c,e_d]$, with $a<b, c<d$, $a,b,c,d \in \{1,\ldots,p \}$. If $D_{x,y}$ are the entries of the matrix of $D$ in the chosen basis, we use the definition of the bracket in $\ngo_{\gca}$, the $\ip_1$-orthogonality of the basis, and the fact that $D$ is a derivation, to obtain the formula
\be
\frac{\la D[e_a,e_b], [e_c,e_d] \ra_1 }{ \|\,[e_c,e_d] \|_1}
= \left\{
      \begin{array}{ll}
        D_{d,b} + D_{c,a}, & a=c, b=d \hbox{;} \\
        D_{d,b}, & a=c, b\neq d \hbox{;} \\
        D_{c,a}, & b=d, a\neq c\hbox{;} \\
        D_{c,b}, & a=d \hbox{;} \\
        D_{d,a}, & b=c \hbox{;} \\
        0, & \hbox{otherwise;}
      \end{array}
    \right.
\ee
where $\| v\|_1 = \la v,v\ra_1^{1/2}$. An analogous formula for $\frac{\la D[e_c,e_d], [e_a,e_b] \ra_1 }{ \|\,[e_a,e_b] \|_1}$ can be obtained changing the roles of $a,b$ with $c,d$ respectively. Since $D_{1,1}$ is symmetric as a matrix we obtain at once that
\begin{equation}\label{SimetriaCasi} \frac{\la D[e_c,e_d], [e_a,e_b] \ra_1 }{ \|\,[e_a,e_b] \|_1} = \frac{\la D[e_a,e_b], [e_c,e_d] \ra_1 }{ \|\,[e_c,e_d] \|_1}. \end{equation}
In the cases `$a=c, b=d$' or `otherwise' we obtain immediately (\ref{Simetria}). The statement is also trivial if $\la D[e_c,e_d], [e_a,e_b] \ra_1 = 0$.  Suppose then that we have $b=c$ and $\la D[e_c,e_d], [e_a,e_b] \ra_1 \neq 0$ (the rest of the cases are similar). This amounts to say that $D_{d,a} \neq 0$. Since $D_{11} \in \bigoplus_{\lambda \in \Lambda} \sym(V_{\lambda})$, $a$ and $d$ are vertices of the same coherent component. Then $ab$ and $cd$ join the same pair of coherent components $S_{\lambda}$ and $S_{\mu}$. Finally, Proposition \ref{mismos_pesos} and formula (\ref{pinilsoliton}) imply that $\ip_1$ restricted to $[V_{\lambda}, V_{\mu}]$ is a multiple of the canonical inner-product there, therefore $\|\,[e_a,e_b] \|_1 = \|\,[e_c,e_d] \|_1$. Hence (\ref{SimetriaCasi}) implies (\ref{Simetria}) and we are done.
\end{proof}

\vskip12pt

\subsection{The group $\Aut({\ngo_{\gca}}) \cap \Or({\ngo_{\gca}})$}\label{autOrSec}
Following \cite{DM} we write $$\Aut(\ngo_{\gca}) = T \ltimes U,$$
where $U = \{\tau \in \Aut(\ngo_{\gca}) : \tau = \left[ \bsm I_p & 0 \\ \theta & I_q\\ \esm \right], \theta \in \Hom(V,[V,V]) \}$, and $T = \{ \tau \in \Aut(\ngo_{\gca}) : \tau(V) = V \} $, identifying matrices with operators using the canonical basis of $\ngo_{\gca}$ fixed, ordered as in (\ref{matrizggo}) (this decomposition is analogous to the one in Proposition \ref{descDer} for $\Der(\ngo_{\gca})$). Take  $\tau \in \Aut({\ngo_{\gca}}) \cap \Or({\ngo_{\gca}})$, $\tau = \tau_{\theta} \tau_T$ with $\tau_{\theta} = \left[ \bsm I_p & 0\\ \theta & I_q\\ \esm \right] \in U$ ($\theta \in \Hom(V, [V,V])$) and $\tau_T = \left[ \bsm \tau_{11} & 0\\ 0 & \tau_{22} \\ \esm \right] \in T$. We see that $\tau = \left[ \bsm \tau_{11} & 0\\ \theta \tau_{11} & \tau_{22}\\ \esm \right]$, hence the orthogonality condition, together with the fact that the basis is $\ip_1$-orthogonal and that $\tau_{11}, \tau_{22}$ are invertible, imply $\theta = 0$. We therefore have
\begin{equation}\label{autOr_subset_T}
\Aut({\ngo_{\gca}}) \cap \Or({\ngo_{\gca}}) \subseteq T.
\end{equation}

Recall the discussion after Proposition \ref{descDer}: using (\ref{autOr_subset_T}), it restricts our attention to the Lie subgroup $G = \{\tau|_V : \tau \in T \}$. We know that an element of $G$ determines a unique extension to $\Aut(\ngo_{\gca})$, and for this to be $\ip_1$-orthogonal it is necessary that the former belongs to $\Or(V,\ip)$. In spite of the fact that this condition is not sufficient, we shall show that actually it is, in some special cases.

For each $e_{p+k} \in [V,V]$, if $l_k = e_i e_j$ with $1\leq i<j\leq p$, we define the $n\times n$ antisymmetric matrix $J(e_{p+k}) = E_{i,j} - E_{j,i}$, where $E_{a,b}$ is the matrix with its $a,b$-entry equal to 1, and zeroes otherwise. By linearity we can define $J(z)$ for every $z\in [V,V]$. Then the canonical inner product in $[V,V]$ is given by the formula
\be \la z,w \ra = -\frac12 \tr J(z) J(w), \qquad z,w \in [V,V]. \ee
Moreover, these matrices help us relating the Lie bracket with that inner product, via the following identity
\be \la [v,w],z \ra = \la J(z)v,w \ra, \qquad v,w\in V, \,\, z\in [V,V]. \ee

\begin{lemma}\label{extensionOr}
Given $A\in G \cap \Or(V, \ip)$, $A = \tau |_V$ with $\tau \in T$, we always have that $\tau \in \Or(\ngo_{\gca}, \ip)$. And $\tau \in \Or(\ngo_{\gca}, \ip_1)$ if and only if $\tau$ commutes with $g^2$ (where $g$ is the matrix defined in (\ref{def_g})).
In particular, this happens if $A$ leaves the coherent components invariant.
\end{lemma}
\begin{proof}
Take $A$ and $\tau$ as in the statement. Via identification of $\tau$ with a matrix, we see that the first $p\times p$ block of it is precisely $A$. From
\begin{eqnarray*}
\la J(\tau^t z) v, w \ra = \la [v,w],\tau^t(z) \ra = \la \tau [v,w], z \ra = \la [\tau v, \tau w], z \ra = \\
= \la J(z) \tau v, \tau w \ra = \la \tau^t J(z) \tau v ,w \ra, \,\,\,\,\,\,
\forall v,w \in V, \,\, z\in [V,V]
\end{eqnarray*}
we have that $J(\tau^t z) = \tau^t J(z) \tau$, where the transposes are with respect to $\ip$. If $\tilde{A} = \left[ \bsm A & 0\\ 0 & I_q\\ \esm \right]$ this implies $J(\tau^t z) = \tilde{A}^t J(z) \tilde{A}$. And since $\tilde{A}$ is $\ip$-orthogonal we have
\be \|z\|^2 = -\frac12 \tr J(z)^2 = -\frac12 \tr J(\tau^t z)^2 = \|\tau^t z\|^2,\,\,\,\,\, \forall z\in [V,V] \ee
hence $\tau^t$ is orthogonal, and so is $\tau$.

We now turn to the case of $\ip_1$-orthogonality. Let us suppose first that $\tau$ commutes with $g^2$; this amounts to say that $\tau^t g^{-2} \tau = g^{-2}$ since we have already proved that $\tau$ is $\ip$-orthogonal. Hence,
\begin{eqnarray*}
\la \tau z, \tau w \ra_1 = \la g^{-1} \tau z , g^{-1} \tau w \ra = \la \tau ^t g^{-2} \tau z ,w \ra = \\
= \la g^{-2} z,w \ra = \la g^{-1}z, g^{-1}w \ra = \la z,w \ra_1, \,\,\,\,\, \forall z,w \in [V,V]
\end{eqnarray*}
and consequently $\tau \in \Or(\ngo_{\gca}, \ip_1)$. Conversely, if $\tau$ is $\ip_1$-orthogonal then
\be
\la \tau^t g^{-2} \tau z,w \ra = \la \tau z, \tau w \ra_1 = \la z,w \ra_1 = \la g^{-2}z,w \ra, \,\,\,\, \forall z,w\in [V,V]
\ee
thus $\tau^t g^{-2} \tau = g^{-2}$ and therefore $\tau$ commutes with $g^2$.

In order to prove our last assertion suppose that $A$ has the given condition, and let us rewrite the identity $\tau g^2 = g^2 \tau$ as $\tau_{i,j} (g^2_{i,\,i} - g^2_{j,\,j}) = 0, \,\,\, \forall\, i,j\in \{1,\ldots,n \}, $
which is equivalent to
\begin{equation}\label{tau_g^2}
\tau_{p+i,p+j}(c_i - c_j) = 0,\,\,\,  \forall\, i,j\in \{1,\ldots,q \}.
\end{equation}
Now take $i,j$ such that $\tau_{p+i,p+j} \neq 0$. Since $A$ leaves the coherent components invariant, the edges $e_{p+i}$ and $e_{p+j}$ are similar. Hence from Proposition \ref{mismos_pesos} we have that $c_i = c_j$, so (\ref{tau_g^2}) holds, and this finishes the proof.
\end{proof}

In \cite{DM} a characterization for $G^0$, the identity component of the Lie group $G$, is given, according to the one obtained in Proposition \ref{ggo} for its Lie algebra $\ggo$.
\begin{proposition}\cite{DM}
Following the notation from Proposition \ref{ggo}, we have that
\be G^0 = (\prod_{\lambda \in \Lambda} \Gl^+(V_{\lambda})) \cdot M \ee
where, for each $\lambda \in \Lambda$, $\Gl^+(V_{\lambda})$ denotes the subgroup of $\Gl(V_{\lambda})$ consisting of elements with positive determinant (the product is viewed as a subgroup of $\Gl(V)$ via the canonical inclusion) and $M$ is a closed, connected, normal, nilpotent Lie subgroup of $G_0$. Furthermore, the elements of $\Lambda$ can be enumerated as $\lambda_1, \ldots, \lambda_k$ so that $\bigoplus_{i\leq j} V_{\lambda_i}$ is $G^0$-invariant, for every $ j = 1, \ldots, k$.
\end{proposition}

As in (\ref{matrizggo}), we see that the canonical matrix representations of the elements of $G^0$ are block-triangular:
\begin{equation}\label{matrizG}
\begin{array}{c}
    \lambda_1 \\
    \lambda_2 \\
    \lambda_3 \\
              \\
    \lambda_k \\
\end{array}
\left[
  \begin{array}{ccccc}
      B_1   &   *   &    *  & \dots &    *  \\
       0    &   B_2 &    *  & \dots &   *   \\
       0    &   0   &  B_3  & \dots &   *   \\
      \vdots& \vdots&\vdots & \ddots&\vdots \\
       0    &   0   &   0   & \dots &  B_k  \\
  \end{array}
\right]
\end{equation}
where $B_i \in \Gl^+(V_{\lambda_i})$, $i=1,\ldots, k$.

It is clear now that for an element of $G^0$ to be $\ip$-orthogonal it is necessary that its $M$-component be trivial, and that $B_i\in \Or(V_{\lambda_i},\ip)$ $\forall i$. Since $\det(B_i) > 0$ the last assertion is equivalent to $B_i\in SO(V_{\lambda_i},\ip)$. Conversely, it is easily seen that these conditions are sufficient, hence
\begin{equation}\label{G0capOr}
G^0 \cap \Or(V,\ip) = \prod_{i=1}^k \SO(V_{\lambda_i},\ip)
\end{equation}

Let $D(V) \subseteq GL(V)$ be the set of operators whose canonical matrix representations are diagonal; it is easy to check that $D(V) \subseteq G$. Since we are mainly interested in the orthogonal elements, we define $D_1(V) = D(V)\cap \Or(V,\ip) = \{ A\in D(V) : (A)_{ii} \in \{-1,1 \}\,\, \forall i= 1, \ldots, p \}$.

Finally, let $\Sigma(S,E)$ be the set of graph automorphisms of $(S,E)$. By linearity, every element of $\Sigma(S,E)$ can be uniquely extended to an operator in $GL(V)$, and this operator can be extended to an element in $\Aut(\ngo_{\gca})$; we also call $\Sigma(S,E)$ the discrete subgroup of $GL(V)$ thus obtained. Since a permutation is orthogonal, we have that $\Sigma(S,E) \subseteq G\cap \Or(V,\ip)$.

We conclude this section with the main theorem about $\ip_1$-orthogonal automorphisms of $\ngo_{\gca}$.

\begin{theorem}\label{TeoAutOr}
$\Aut(\ngo_{\gca}) \cap \Or(\ngo_{\gca},\ip_1) = T\cap\Or(\ngo_{\gca},\ip_1)$, where $T$ is the set of automorphisms of $\ngo_{\gca} = V + [V,V]$ that leave $V$ invariant. Every element of $T$ is determined by its restriction to $V$; we call $G$ the set of those restrictions. For $G_{or}$, the subgroup of $G$ consisting of restrictions of elements from $T\cap\Or(\ngo_{\gca},\ip_1)$ to $V$, the following properties hold:
\begin{itemize}
\item[(i)] If $G^0$ is the identity component of $G$,
    \be G^0 \cap \Or(V,\ip) = \prod_{i=1}^k \SO(V_{\lambda_i},\ip) \subseteq G_{or}, \ee

\item[(ii)] $G_{or}$ contains $D_1(V)$ and $\Sigma(S,E)$ as subgroups.

\item[(iii)] $G_{or}$ also contains the operators that act $\ip$-orthogonally on the coherent components of $(S,E)$, i.e.
    \be \prod_{i=1}^k \Or(V_{\lambda_i},\ip) \subseteq G_{or},\ee
\end{itemize}
where each $\SO(V_{\lambda_i},\ip)$ and $\Or(V_{\lambda_i},\ip)$ is viewed as a subgroup of $\Gl(V)$ via the canonical embedding.
\end{theorem}

\begin{proof}
The first identity is (\ref{G0capOr}). In order to prove the inclusion, we notice that an automorphism $\tau$ obtained from an element of $\prod_{i=1}^k \SO(V_{\lambda_i},\ip)$ leaves the coherent components invariant. Hence (i) follows from Lemma \ref{extensionOr}.

An element from $D_1(V)$ gives rise to a (canonical basis) diagonal automorphism, which obviously commutes with $g^2$, and then by Lemma \ref{extensionOr} it is $\ip_1$-orthogonal. To see that $\Sigma(S,E) \subseteq G_{or}$, first take an automorphism $\tau$ obtained from an element $\sigma \in \Sigma(S,E)$. From Lemma \ref{extensionOr} it is sufficient to prove that $\tau_{p+i,p+j} (c_i - c_j) = 0$,   $\forall i,j = 1,\ldots,q$. Let us take $i,j$ and suppose that $\tau_{p+i,p+j} \neq 0$. Since $\tau$ is also a permutation, we have that $\tau(e_{p+j}) = e_{p+i}$. This implies that there exists a graph automorphism of $(S,E)$ that maps the edge $l_i$ to the edge $l_j$. Therefore, by the uniqueness of the weights $\{c_r\}$ that make the graph positive, we must have that $c_i = c_j$ and our identity holds.

Finally, let us write $D_1(V) = \prod_{i=1}^k D_1(V_{\lambda_i})$. Then by (i) and (ii),
\be  \prod_{i=1}^k D_1(V_{\lambda_i}) \SO(V_{\lambda_i},\ip)  \subseteq G_{or}. \ee
The subgroup $\SO(V_{\lambda_i},\ip)$ of $\Or(V_{\lambda_i},\ip)$ has index 2. Since $D_1(V_{\lambda_i}) \nsubseteq \SO(V_{\lambda_i},\ip)$ this forces $D_1(V_{\lambda_i})\SO(V_{\lambda_i},\ip) = \Or(V_{\lambda_i},\ip)$, and the conclusion follows.
\end{proof}

%\newpage
%---------------------------------------------------------------------------------------------------------------------
\section{Description of solsolitons associated with positive graphs}\label{desc}

We are now in a position to approach Problem \ref{problema}. Let $\ag\subseteq \Der(\ngo_{\gca})$ be an abelian Lie subalgebra of $\ip_1$-symmetric derivations of $\ngo_{\gca}$. According to Theorem \ref{TeoDerSim}, the elements of $\ag$ leave $V$ invariant, and are determined by their restrictions to that subspace. Furthermore, if $\bg = \{\tau|_V : \tau\in \ag\}\subseteq \ggo_{sym}$, we know that $ \bg \subseteq \bigoplus_{i =1}^k \sym(V_{\lambda_i},\ip)$.
Let us call $\bg_i = \{ A|_{V_{\lambda_i}} : A\in \bg\}$, $i=1,\ldots,k$. Then we have the decomposition $ \bg = \bigoplus_{i=1}^k \bg_i $. Also, let $\dg_{\ngo_{\gca}} = \dg(\ngo_{\gca}) \subseteq \End(\ngo_{\gca})$ be the Lie algebra of (canonical basis) diagonal operators on $\ngo_{\gca}$.

\begin{proposition}\label{Diagonales}
Every abelian Lie algebra $\ag$ of $\ip_1$-symmetric derivations of $\ngo_{\gca}$ is equivalent (in the sense of Problem \ref{problema}) to a Lie algebra $\ag'\subseteq \dg_{\ngo_{\gca}}$.
\end{proposition}
\begin{proof}
Fix $i\in \{1,\ldots,k\}$. Since $\bg_i$ is a set of $\ip$-symmetric, pairwise commuting operators, it is simultaneously diagonalizable, meaning that there exists $r_i\in \Or(V_{\lambda_i},\ip)$ such that $r_i \bg_i r_i^{-1} \subseteq \dg(V_{\lambda_i})$. Now let $r\in \prod_{i=1}^k \Or(V_{\lambda_i}, \ip)$ given by
\be
r = \left[
      \bsm
        r_1 & 0 & \dots & 0 \\
        0 & r_2 & \dots& 0 \\
        \vdots & \vdots & \ddots & \vdots \\
        0 & 0 & \dots & r_k \\
      \esm
    \right].
\ee
It is clear that $r \bg r^{-1} \subseteq \dg(V)$. According to Theorem \ref{TeoAutOr}, there exists $h\in \Aut(\ngo_{\gca})\cap \Or(\ngo_{\gca}, \ip_1)$ such that $h|_V = r$; the proof is completed if we show that $h \ag h^{-1} \subseteq \dg_{\ngo_{\gca}}$. Take $D\in \ag$. Then since $h$ is a Lie automorphism, $h D h^{-1}$ is also a derivation. And this, together with the fact that $(h D h^{-1})|_V = r (D|_V) r^{-1}$ is diagonal, imply that $h D h^{-1}$ is also diagonal (in the canonical basis). This completes the proof.
\end{proof}

\begin{remark}
The fact that $h$ is an $\ip_1$-orthogonal automorphism implies that the Lie algebra $\ag' = h\ag h^{-1}$ is also abelian, and consists of $\ip_1$-symmetric derivations of $\ngo_{\gca}$.
\end{remark}

This reduces our problem to classify subalgebras of $\dg_{\ngo_{\gca}}$, since every equivalence class has a representative  of that kind. Then we have to restrict our attention to the operators $h$ that conjugate $\dg_{\ngo_{\gca}}$ onto itself. Thanks to the following lemma, we can assume that $h$ is a permutation operator.

\begin{lemma}\label{bastaConPerm}
Let $\ag_1 \subseteq \dg_{\ngo_{\gca}}$ be an abelian Lie algebra of $\ip_1$-symmetric derivations of $\ngo_{\gca}$. If $h\in \Aut(\ngo_{\gca})\cap \Or(\ngo_{\gca},\ip_1)$ is such that $\ag_2 = h \ag_1 h^{-1}  \subseteq \dg_{\ngo_{\gca}}$, then there exists a permutation $P \in \Aut(\ngo_{\gca})\cap \Or(\ngo_{\gca},\ip_1)$ such that $\ag_2 = P\ag_1 P^{-1}$.
\end{lemma}
\begin{proof}
The group $H = \{h\in \Aut(\ngo_{\gca}) : h^t \in  \Aut(\ngo_{\gca}) \}$ is a real reductive algebraic group (see \cite[Ch. 2]{W}, or \cite[Ch.VII, $\S 2$]{K}). The Cartan decomposition of its Lie algebra $\hg = \{ A\in \Der(\ngo_{\gca}) : A^t \in \ngo_{\gca} \}$ is given by $\hg = \kg \oplus \pg$, where $\kg= \Der(\ngo_{\gca})\cap\so(\ngo_{\gca},\ip_1)$ and $\pg=\Der(\ngo_{\gca})\cap\sym(\ngo_{\gca},\ip_1)$. The subgroup $K = \Aut(\ngo_{\gca}) \cap \Or(\ngo_{\gca},\ip_1)$ is a maximal compact subgroup of $H$, with Lie algebra $\kg$. It is clear that $\dg_{\ngo_{\gca}}$ is a maximal abelian Lie subalgebra, and by Proposition \ref{Diagonales} every such Lie algebra is obtained by conjugacy of $\dg_{\ngo_{\gca}}$. If we fix $\dg_{\ngo_{\gca}}$, the restricted roots, the Weyl chambers and the Weyl group $W = N_K(\dg_{\ngo_{\gca}})/Z_K(\dg_{\ngo_{\gca}})$ are defined. $W$ is a finite group, and acts simply transitively on the set of Weyl chambers. We also know that if two Lie subalgebras $\ag_1, \ag_2 \subseteq \dg_{\ngo_{\gca}}$ are conjugate by an element of $K$, then they are also conjugate by an element of $W$ (see \cite[Ch.VII, Prop. 2.2]{H}). The proof will be finished if we show that $W$ is the group of permutations that are in $K$.

It is clear that $Z_K(\dg_{\ngo_{\gca}}) = \dg_{\ngo_{\gca}}$. Now take $T\in N_K(\dg_{\ngo_{\gca}})$, and $D_1\in \dg_{\ngo_{\gca}}$ having pairwise different entries in its diagonal. From $T D_1 T^{-1}\in \dg_{\ngo_{\gca}}$ it follows that the (canonical basis) matrix representation of $T$ has at most one non-negative entry in each row and column. Hence, $N_K(\dg_{\ngo_{\gca}}) = \{PD : P\in K \,\hbox{permutation, }\, D\in \dg_{\ngo_{\gca}} \}$. Therefore $W$ is precisely the group we wanted, and this finishes the proof.
\end{proof}

Not every permutation $P$ is a $\ip_1$-orthogonal automorphism of $\ngo_{\gca}$. The following lemma gives us the conditions $P$ must satisfy.

\begin{lemma}\label{bastaConAut}
Given a permutation $P\in \End(\ngo_{\gca})$, we have that $P\in \Aut(\ngo_{\gca})\cap \Or(\ngo_{\gca},\ip_1)$ if and only if $P$ leaves $V$ invariant, and acts on the canonical basis of that subspace as a graph automorphism.
\end{lemma}
\begin{proof}
By definition of the Lie bracket in $\ngo_{\gca}$, the necessity is clear. On the other hand, the converse is precisely Theorem \ref{TeoAutOr}, (ii).
\end{proof}

We are now in a position to prove the main result of this paper. To simplify notation, we call $\diag(v) = \diag(v_1,\ldots,v_m)$ the operator in $\End(\RR^m)$ whose matrix in canonical basis is diagonal, and its diagonal entries are the $v_i$. It is clear that $\diag : \RR^m \rightarrow \dg_m$ is a linear isomorphism. Also, if $\sigma\in \SSS_m$ and $v\in \RR^m$, we denote $\sigma(v) = (v_{\sigma(1)}, \ldots, v_{\sigma(m)})$. Clearly, $\diag(\sigma(v)) = P \diag(v) P^{-1}$, where $P\in \End(\RR^m)$ is the permutation operator associated with $\sigma$.

\begin{theorem}\label{TeoMain}
Given a positive graph $\gca = (S,E)$ with $p$ vertices and $q$ edges, the solsolitons associated with it are parameterized by the vector subspaces of $\RR^p$, and two such subspaces $\sca$, $\sca'$ give rise to isometric solsolitons if and only if there exists $\sigma \in \Aut(\gca)$ such that $\sigma(\sca) = \sca'$. The parametrization is according to the following construction:

If $\ngo_{\gca}$ is the 2-step nilpotent Lie algebra associated with $\gca$, in which a unique (up to orthogonal automorphisms and scaling) nilsoliton metric $\ip_1$ can be defined, then for each subspace $\sca \subseteq \RR^p$ we define a solsoliton $S$ with Lie algebra $\sg = \ag \oplus \ngo_{\gca}$ as in Proposition \ref{constr_solsolitons}, where $\ag$ is the abelian Lie algebra of $\ip_1$-symmetric derivations of $\ngo_{\gca}$ obtained from $\bg = \diag(\sca) \subseteq \End(V)$.
\end{theorem}
\begin{proof}
Given the Lie algebra $\ag$, using Proposition \ref{Diagonales} we may suppose without any lose of generality that $\ag \subseteq \dg_{\ngo_{\gca}}$. $\ag$ is determined by the restrictions of its elements to $V$, and this set of restrictions $\bg \subseteq \dg_{p}$ is in itself a Lie algebra of operators of $\End(V)$. It is clear that $\bg = \diag(\sca)$, for some subset $\sca \subseteq \RR^p$. Since $\diag$ is a linear isomorphism, $\sca$ is a vector subspace.

Now if $\ag, \ag' \subseteq \dg_{\ngo_{\gca}}$ are two Lie algebras of $\ip_1$-symmetric derivations of $\ngo_{\gca}$, from Proposition \ref{EquivSolsolitons} they give rise to isometric solsolitons if and only if there exists $h\in \Aut(\ngo_{\gca})\cap \Or(\ngo_{\gca},\ip_1)$ such that $\ag = h\ag h^{-1}$. Using Lemmas \ref{bastaConPerm} and \ref{bastaConAut} we may suppose that $h = P_{\sigma}$, where $P_{\sigma}$ is a permutation associated to a graph automorphism $\sigma$. We write $\ag = \diag(\sca)$, $\ag' = \diag(\sca')$, and then the conjugacy condition is rewritten as $\sigma(\sca) = \sca'$. The proof is now complete.
\end{proof}

\begin{example} We end this section by developing in detail all what has been seen in this paper for the following graph $\gca$:
\setlength{\unitlength}{.3cm}

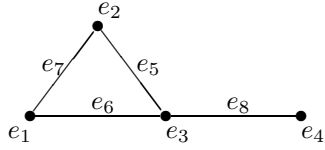
\begin{picture}(9,3.5)(-13,2)
\put(0.15,0.2){\line(3,4){2.65}}  %1-2
\put(3.15,3.8){\line(3,-4){2.65}} %2-3
\put(0.25,0){\line(1,0){5.5}}      %3-1
\put(6.25,0){\line(1,0){5.5}}
\put(0,0){\circle*{.5}} %V1
\put(3,4){\circle*{.5}} %V2
\put(6,0){\circle*{.5}} %V3
\put(12,0){\circle*{.5}} %V3
\put(-1,-1){$e_1$}
\put(3,4.5){$e_2$}
\put(6,-1){$e_3$}
\put(12,-1){$e_4$}
\put(4.7,2){$e_5$}
\put(2.7,.3){$e_6$}
\put(.5,2){$e_7$}
\put(8.7,.3){$e_8$}
\\

\end{picture}
\vskip30pt

We have labeled the vertices and edges according to the elements they represent in the canonical basis of $\RR^8$. The graph $\gca$ is clearly positive, with weights
\setlength{\unitlength}{.3cm}
\begin{picture}(10,5.5)(-10,-1)
\put(0.15,0.2){\line(3,4){2.65}}  %1-2
\put(3.15,3.8){\line(3,-4){2.65}} %2-3
\put(0.25,0){\line(1,0){5.5}}      %3-1
\put(6.25,0){\line(1,0){5.5}}
\put(0,0){\circle*{.5}} %V1
\put(3,4){\circle*{.5}} %V2
\put(6,0){\circle*{.5}} %V3
\put(12,0){\circle*{.5}} %V3
%\put(-1,-1){$e_1$}
%\put(3,4.5){$e_2$}
%\put(6,-1){$e_3$}
%\put(12,-1){$e_4$}
\put(5,2){$\frac16$}
\put(2.7,.5){$\frac16$}
\put(0.5,2){$\frac13$}
\put(8.7,.5){$\frac13$}
\end{picture}
%\vskip30pt

That is, $\textbf{c} = \left(\frac16,\frac16,\frac13,\frac13\right)$ is a solution to (\ref{positivo}) with $\sum_{i=1}^4 c_i = 1$ ($c_i$ represents the weight corresponding to the edge $e_{i+4}$). Therefore we can construct a nilsoliton $\ngo_{\gca} = (\RR^8, [\cdot,\cdot], \ip_1)$ associated with $\gca$. According to (\ref{def_g}), the nilsoliton inner product $\ip_1$ is given by
$$\la x,x \ra_1 = x_1^2 + x_2^2 + x_3^2 + x_4^2 + 6 x_5^2 + 6x_6^2 + 3x_7^2 + 3x_8^2, \qquad x\in \RR^8.$$

It is easy to see that the coherent components of the graph are precisely $S_1 = \{e_1,e_2 \}$, $S_2 = \{e_3\}$, $S_3 = \{ e_4\}$, and that the automorphism group is $\Aut(\gca) \simeq \ZZ_2$ (we can only interchange $e_1$ and $e_2$).

Regarding the nilsoliton $\ngo_\gca$, we have the following information: $\Der(\ngo_\gca)$ can be decomposed as $\ug \oplus \tgo$, according to Proposition \ref{descDer}. Clearly, $\ug \simeq \End(\RR^4)$, and $\tgo$ is determined by (and isomorphic to) the Lie algebra $\ggo$, restricting the operators to $V$. Fixing the ordered basis $\{e_4,e_1,e_2,e_3\}$ for $V$, it can be seen that
$$\ggo = \left\{
\left[
      \bsm
        a_{11} & 0      & 0      & 0        \\
        0      & a_{22} & a_{23} & 0        \\
        0      & a_{32} & a_{33} & 0        \\
        0      & 0      & 0      & a_{44}   \\
      \esm
    \right] : a_{ij} \in \RR
 \right\} \oplus
 \left\{
\left[
      \bsm
        0      & a_{12} & a_{13} & a_{14}   \\
        0      & 0 & 0 & a_{24}        \\
        0      & 0 & 0 & a_{34}        \\
        0      & 0      & 0      & 0   \\
      \esm
    \right] : a_{ij} \in \RR
 \right\},
 $$
is the decomposition $\ggo = \qg \oplus \mg$ from Proposition \ref{ggo}. It is now clear that the set $\ggo_{sym}$, that determines the $\ip_1$-symmetric derivations of $\ngo_\gca$, is precisely
$$ \ggo_{sym} = \left\{
\left[
      \bsm
        a_{11} & 0      & 0      & 0        \\
        0      & a_{22} & a_{23} & 0        \\
        0      & a_{32} & a_{33} & 0        \\
        0      & 0      & 0      & a_{44}   \\
      \esm
    \right] : a_{ij} \in \RR, a_{23} = a_{32}
 \right\},$$
as stated in Theorem \ref{TeoDerSim}. Similarly, we can compute the automorphism group $\Aut(\ngo_{\gca})$. We know from Section \ref{autOrSec} that $\Aut(\ngo_\gca) = T \ltimes U$, with $U \simeq \End(\RR^4)$. Moreover, $T \simeq G$ (taking restrictions to $V$), and in terms of the basis previously fixed for $V$ we see that
$$G = \left\{
\left[
      \bsm
        a_{11} & 0         & 0        \\
        0      & B & 0             \\
        0      & 0         & a_{44}   \\
      \esm
    \right] : a_{11}, a_{44} \neq 0,  B\in \Gl(2,\RR)
 \right\} \ltimes
 \left\{
\left[
      \bsm
        1      & a_{12} & a_{13} & a_{14}   \\
        0      & 1 & 0 & a_{24}        \\
        0      & 0 & 1 & a_{34}        \\
        0      & 0      & 0      & 1   \\
      \esm
    \right] : a_{ij} \in \RR
 \right\},
 $$
hence
$$G\cap \Or(V,\ip) = \left\{
\left[
      \bsm
        a_{11} & 0         & 0        \\
        0      & B & 0             \\
        0      & 0         & a_{44}   \\
      \esm
    \right] : a_{11}, a_{44} \in \{\pm 1\},  B\in \Or(2,\RR)
 \right\},
$$
so using Theorem \ref{TeoAutOr} we can conclude that $G_{or} = G\cap \Or(V,\ip)$.

In order to study solsolitons obtained from the nilsoliton $\ngo_{\gca}$ we would have to find the abelian Lie algebras of $\ip_1$-symmetric derivations of $\ngo_{\gca}$. And to classify them up to conjugation by an $\ip_1$-orthogonal automorphism of $\ngo_{\gca}$, in the light of Theorem \ref{TeoMain} we are reduced to consider the vector subspaces of $\RR^4$ modulo $\ZZ_2$ (with $\ZZ_2$ acting on the subspaces of $\RR^4$ by permutating the first two coordinates, taking $\{e_1, e_2,e_3,e_4 \}$ as the ordered basis).

For instance, if $\dim \ag = 1$, we have to consider one-dimensional subspaces $\sca$ of $\RR^4$. We may write $\sca = \RR v$, for some $v\in \RR^4$. $v$ gives rise to a symmetric derivation $A$ of $\ngo_{\gca}$ whose matrix with respect to the canonical basis is $\diag(v_1,v_2,v_3,v_4, v_2+v_3, v_1+v_3, v_1+v_2, v_3+v_4)$. The solsoliton thus obtained is $\sg = \RR A \oplus \ngo_{\gca} $ (semidirect product), where the bracket is extended by letting $\ad A$ act as $A$ on $\ngo_{\gca}$. The inner product is defined according to Proposition \ref{constr_solsolitons}.

In the case $\dim \ag = 2$, the isometry classes of solsolitons associated with $\gca$ are parameterized by $\Gr_2(\RR^4)/\ZZ_2$; we obtain thus a $4$-parameter family of non-isometric solsolitons (recall that $\dim \Gr_k(\RR^n) = k(n-k)$). As in the first case, given $\sca \in \Gr_2(\RR^4)$ we take a basis $\{ v_1, v_2\}$ of it, and $v_i$ determines a diagonal derivation $A_i$ of $\ngo_{\gca}$. The solsoliton is then $\sg =  \RR A_1 \oplus \RR A_2 \oplus \ngo_{\gca} $, with $\ad A_i |_{\ngo_\gca} = A_i$.

Finally, the case $\dim \ag = 3$ is very similar to those mentioned before. And the cases $\dim \ag = 0$ or $4$ give rise to a unique solsoliton (up to isometry) each, since we have no choice for the subspace.

\end{example}

%\newpage
%---------------------------------------------------------------------------------------------------------------------
\section{Some examples of positive graphs}\label{ejemplos}

In this section we aim to classify connected graphs with up to three coherent components, according to positivity. As a consequence, we obtain many new examples of positive graphs, leading to examples of solsolitons as well. For other examples of positive and non-positive graphs we refer the reader to \cite{LauWill}.

This classification relies mainly on Proposition \ref{mismos_pesos}, and on the idea of representing a graph by its coherence graph. For example, let $\gca$ be a connected graph having one coherent component. According to the properties mentioned about the coherent components, we see that $\gca$ is a complete graph. Then it is clearly positive, a solution to (\ref{positivo}) being the vector $(1,\ldots,1)$. Moreover, $(1,\ldots,1)$ is a solution to (\ref{positivo}) if and only if every edge of the graph has the same number of adjacent edges, that is, its line graph is a \emph{regular} graph.

We represent a graph via its coherence graph. Each circle represents a coherent component, being black if the correspondent component is a complete graph, and white if it is discrete. The existence of an edge joining two circles represents the fact that every edge joining vertices between those coherent components is present in the original graph. Finally, the natural number near to each coherent component is the number of vertices that it contains.

\subsection{Two coherent components} We have the following cases:

\setlength{\unitlength}{.2cm}

\subsubsection{}
\begin{picture}(9,1)(-2,-1)
%\put(0.3,0.4){\line(3,4){2.4}}  %1-2
%\put(3.3,3.6){\line(3,-4){2.4}} %2-3
\put(0.5,0){\line(1,0){5}}      %3-1
\put(0,0){\circle{1}} %V1
%\put(3,4){\circle{1}} %V2
\put(6,0){\circle{1}} %V3
\put(-2,-1){r}
\put(7.5,-1){s}
%\put(3,5.5){t}
\end{picture}

This is the case of a complete bipartite graph. It is positive for every $r,s$, since it is regular.

\subsubsection{}
\begin{picture}(9,1)(-2,-1)
%\put(0.3,0.4){\line(3,4){2.4}}  %1-2
%\put(3.3,3.6){\line(3,-4){2.4}} %2-3
\put(0.5,0){\line(1,0){5}}      %3-1
\put(0,0){\circle{1}} %V1
%\put(3,4){\circle{1}} %V2
\put(6,0){\circle*{1}} %V3
\put(-2,-1){r}
\put(7.5,-1){s}
%\put(3,5.5){t}
\end{picture}

We denote $S_1$, $S_2$ the coherent components with $r$, $s$ vertices, respectively. By Proposition \ref{mismos_pesos}, there are only two possibly different edge weights in this case: $a$, for the edges joining $S_1$ with $S_2$, and $b$, for the edges inside $S_2$. In the system (\ref{positivo}), equations corresponding to similar edges are identical. Hence we can rewrite (\ref{positivo}) as
\be \left\{
\begin{array}{r}
(r+s+1)a + (s-1)b = \nu \\
2ra + (2s-1)b = \nu
\end{array}
\right. \ee
If we call $A$ the matrix of the system, we have $\det(A) >0 $ for every $r$, $s$. Fix any $\nu >0$. The solution is given by
\be \left[ \begin{array}{c} a \\  b \\ \end{array} \right] = \frac{\nu}{\det(A)} \left[ \begin{array}{c} s \\  1-r+s \\ \end{array} \right]
\ee
and $a,b$ are positive if and only if $1-r+s>0$. Therefore, $\gca$ is positive if and only if $s\geq r$.

\begin{remark}
Note that to avoid repetition we omit the case `black-black', since those graphs actually have one coherent component. For the same reason, we will also omit many cases when studying the case of $3$ components.
\end{remark}

\subsection{Three coherent components} Working in a similar way we obtain necessary and sufficient conditions, in terms of the number of vertices of the coherent components, for a graph of this kind to be positive. We present here only one case, of a total of six, the others being completely analogous.

\subsubsection{}
\begin{picture}(9,5.5)(-2,2)
\put(0.3,0.4){\line(3,4){2.4}}  %1-2
\put(3.3,3.6){\line(3,-4){2.4}} %2-3
\put(0.5,0){\line(1,0){5}}      %3-1
\put(0,0){\circle{1}} %V1
\put(3,4){\circle{1}} %V2
\put(6,0){\circle{1}} %V3
\put(-2,-1){r}
\put(7.5,-1){s}
\put(3,5.5){t}
\end{picture} \\
\\

Let us call $S_1$, $S_2$, $S_3$ the coherent components with $r, s, t$ vertices, respectively. This case has only three different weights, $a,b,c$, corresponding respectively to the edges joining $S_1$ and $S_2$, $S_2$ and $S_3$, $S_3$ and $S_1$. Simplifying (\ref{positivo}) we obtain
\be \left\{
\begin{array}{r}
(r+s+1)a + tb + tc = \nu \\
ra + (s+t+1)b + rc = \nu \\
sa + sb + (r+t+1)c = \nu
\end{array}
\right. \ee
hence the matrix of the system is
\be A = \left[
          \begin{array}{ccc}
            r+s+1 & t & t \\
            r & s+t+1 & r \\
            s & s & r+t+1 \\
          \end{array}
        \right]
\ee
and then the solution is the vector
\be
\left[ \begin{array}{c} a \\  b \\ c \\ \end{array} \right] =
\frac{\nu}{\det(A)} \left[
      \begin{array}{c}
        {(1 + r + s - t) (1 + t)}\\
        {(1 + r) (1 - r + s + t)} \\
        {(1 + s) (1 + r - s + t)} \\
      \end{array}
    \right].
\ee
Since $\det(A)>0$ in any case, we have that $a,b,c > 0$ if and only if
\be r+s \geq t,\,\,\,\,\,\,\, s+t \geq r,\,\,\, \hbox{ y } \,\,\,\,\,\, t+r\geq s\ee

\vskip10pt
We thus obtain the classification of connected graphs with at most 3 coherent components. Concerning the non-connected graphs with that number of coherent components, we have to say that the problem is reduced to study each connected component, which have less coherent components and therefore lie in the cases we have already considered. There is an exception to this, and it is the case of a discrete and isolated coherent component. However, we do not mention it in the classification because these cases have no edges, and so no system (\ref{positivo}) to solve (still, they are positive). To summarize, we present the obtained information in Table \ref{tabla}.
\setlength{\unitlength}{.15cm}

\begin{table}

\begin{tabular}{|c|c|c|c|c|c|}
\hline
  \textbf{Graph}  & \textbf{CC}  &   \textbf{$p$} &  \textbf{$q$} &  \textbf{Positive iff}   \\
\hline \hline
 \begin{picture}(9,3)(-2,-1)
%\put(0.3,0.4){\line(3,4){2.4}}  %1-2
%\put(3.3,3.6){\line(3,-4){2.4}} %2-3
%\put(0.5,0){\line(1,0){5}}      %3-1
\put(3,0){\circle*{1}} %V1
%\put(3,4){\circle{1}} %V2
%\put(6,0){\circle{1}} %V3
\put(0.5,-1){r}
%\put(7.5,-1){s}
%\put(3,5.5){t}
\end{picture}  &  $1$ &  $r$ &  $\binom{r}{2}$ &  always  \\
\hline
 \begin{picture}(9,3)(-2,-1)
%\put(0.3,0.4){\line(3,4){2.4}}  %1-2
%\put(3.3,3.6){\line(3,-4){2.4}} %2-3
\put(0.5,0){\line(1,0){5}}      %3-1
\put(0,0){\circle{1}} %V1
%\put(3,4){\circle{1}} %V2
\put(6,0){\circle{1}} %V3
\put(-2,-1){r}
\put(7.5,-1){s}
%\put(3,5.5){t}
\end{picture}  &  $2$ &  $r+s$  & $rs$  &  always   \\
\hline
 \begin{picture}(9,3)(-2,-1)
%\put(0.3,0.4){\line(3,4){2.4}}  %1-2
%\put(3.3,3.6){\line(3,-4){2.4}} %2-3
\put(0.5,0){\line(1,0){5}}      %3-1
\put(0,0){\circle{1}} %V1
%\put(3,4){\circle{1}} %V2
\put(6,0){\circle*{1}} %V3
\put(-2,-1){r}
\put(7.5,-1){s}
%\put(3,5.5){t}
\end{picture}  &  $2$ &  $r+s$ &  $rs+\binom{s}{2}$  & $s>r$   \\
\hline
%%%%%%%%%%  CASO 1 %%%%%%%%%
 \begin{picture}(10,5)(-2,0)
\put(0.3,0.4){\line(3,4){2.4}}  %1-2
\put(3.3,3.6){\line(3,-4){2.4}} %2-3
\put(0.5,0){\line(1,0){5}}      %3-1
\put(0,0){\circle{1}} %V1
\put(3,4){\circle{1}} %V2
\put(6,0){\circle{1}} %V3
\put(-2,-0.5){r}
\put(7.5,-0.5){s}
\put(4.5,3.5){t}
\end{picture} &  $3$ &  $r+s+t$  & $rs+st+tr$  &  $ r+s\geq t,\,\, s+t\geq r,\,\, t+r\geq s $ \\
\hline
%%%%%%%%%%  CASO 2 %%%%%%%%%
 \begin{picture}(10,5)(-2,0)
\put(0.3,0.4){\line(3,4){2.4}}  %1-2
\put(3.3,3.6){\line(3,-4){2.4}} %2-3
\put(0.5,0){\line(1,0){5}}      %3-1
\put(0,0){\circle{1}} %V1
\put(3,4){\circle*{1}} %V2
\put(6,0){\circle{1}} %V3
\put(-2,-0.5){r}
\put(7.5,-0.5){s}
\put(4.5,3.5){t}
\end{picture}  &  3 & $r+s+t$ & $rs+st+tr+\binom{t}{2}$  & $1+t > |r-s|$    \\
\hline
%%%%%%%%%%  CASO 3 %%%%%%%%%
 \begin{picture}(10,5)(-2,0)
%\put(0.3,0.4){\line(3,4){2.4}}  %1-2
\put(3.3,3.6){\line(3,-4){2.4}} %2-3
\put(0.5,0){\line(1,0){5}}      %3-1
\put(0,0){\circle{1}} %V1
\put(3,4){\circle*{1}} %V2
\put(6,0){\circle{1}} %V3
\put(-2,-0.5){r}
\put(7.5,-0.5){s}
\put(4.5,3.5){t}
\end{picture}  &  3 &  $r+s+t$ &  $rs+st+\binom{t}{2}$ & ${r + t(1 - r + s)} > 0, \,\,   t+r \geq s$   \\
\hline
%%%%%%%%%%  CASO 4 %%%%%%%%%
 \begin{picture}(10,5)(-2,0)
%\put(0.3,0.4){\line(3,4){2.4}}  %1-2
\put(3.3,3.6){\line(3,-4){2.4}} %2-3
\put(0.5,0){\line(1,0){5}}      %3-1
\put(0,0){\circle{1}} %V1
\put(3,4){\circle*{1}} %V2
\put(6,0){\circle*{1}} %V3
\put(-2,-0.5){r}
\put(7.5,-0.5){s}
\put(4.5,3.5){t}
\end{picture} &  3 & $r+s+t$  &  $rs+st+\binom{s}{2}+ \binom{t}{2}$ &  $( s+t) (s-r) >(r-1)(t-1)$  \\
\hline
%%%%%%%%%%  CASO 5 %%%%%%%%%
 \begin{picture}(10,5)(-2,0)
%\put(0.3,0.4){\line(3,4){2.4}}  %1-2
\put(3.3,3.6){\line(3,-4){2.4}} %2-3
\put(0.5,0){\line(1,0){5}}      %3-1
\put(0,0){\circle*{1}} %V1
\put(3,4){\circle*{1}} %V2
\put(6,0){\circle{1}} %V3
\put(-2,-0.5){r}
\put(7.5,-0.5){s}
\put(4.5,3.5){t}
\end{picture}  &  3 &  $r+s+t$  & $rs+st+\binom{r}{2}+ \binom{t}{2}$  &  $r+t\geq s$   \\
\hline
%%%%%%%%%%  CASO 6 %%%%%%%%%
 \begin{picture}(10,5)(-2,0)
%\put(0.3,0.4){\line(3,4){2.4}}  %1-2
\put(3.3,3.6){\line(3,-4){2.4}} %2-3
\put(0.5,0){\line(1,0){5}}      %3-1
\put(0,0){\circle*{1}} %V1
\put(3,4){\circle*{1}} %V2
\put(6,0){\circle*{1}} %V3
\put(-2,-0.5){r}
\put(7.5,-0.5){s}
\put(4.5,3.5){t}
\end{picture}  &  3 & $r+s+t$  & $rs+st+\binom{r}{2}+ \binom{s}{2}+\binom{t}{2}$  &  always   \\
\hline
\end{tabular}
\vskip5pt
\caption{Positivity of graphs with up to 3 coherent components.}\label{tabla}
\end{table}

%\newpage
%\backmatter
%---------------------------------------------------------------------------------------------------------------------
%\input{misc/biblio}

\end{document}